\documentclass[11pt]{article}
\usepackage{IEEEtrantools}
\usepackage{mathtools, nccmath}
\usepackage{latexsym}
\usepackage{amsfonts}
\usepackage{amssymb}
\usepackage{psfrag}
\usepackage{graphicx}
\usepackage{epsfig}
\usepackage{amsmath,amsfonts,amsthm}
\usepackage{mathrsfs}
\usepackage{enumerate}
\usepackage{dsfont}
\usepackage{hyperref}
\usepackage{pst-node}
\usepackage{cleveref}
\usepackage{enumitem}
\usepackage{tikz-cd}
\usepackage{multirow}
\usepackage{tikz}
\usetikzlibrary{arrows,decorations.markings}
\usepackage{bookmark}
\usepackage{hyperref}
\hypersetup{
     colorlinks   = true,
     citecolor    = blue
}

\newcommand\restr[2]{{
  \left.\kern-\nulldelimiterspace 
  #1 
  \littletaller 
  \right|_{#2} 
  }}

\newcommand{\littletaller}{\mathchoice{\vphantom{\big|}}{}{}{}}

\newcommand{\be}{\begin{equation}}
\newcommand{\ee}{\end{equation}}
\newcommand{\bea}{\begin{eqnarray}}
\newcommand{\eea}{\end{eqnarray}}
\newcommand{\bean}{\begin{eqnarray*}}
\newcommand{\eean}{\end{eqnarray*}}
\newcommand{\brray}{\begin{array}}
\newcommand{\erray}{\end{array}}
\newcommand{\biearray}{\begin{IEEEarray}{rCl}}
\newcommand{\eiearray}{\end{IEEEarray}}

\newcommand{\newsection}[1]{\setcounter{equation}{0}
\setcounter{dfn}{0}
\section{#1}}

\newtheorem{dfn}{Definition}[section]
\newtheorem{thm}[dfn]{Theorem}
\newtheorem{lmma}[dfn]{Lemma}
\newtheorem{ppsn}[dfn]{Proposition}
\newtheorem{crlre}[dfn]{Corollary}
\newtheorem{xmpl}[dfn]{Example}
\newtheorem{rmrk}[dfn]{Remark}

\newcommand{\bdfn}{\begin{dfn}\rm}
\newcommand{\bthm}{\begin{thm}}
\newcommand{\blmma}{\begin{lmma}}
\newcommand{\bppsn}{\begin{ppsn}}
\newcommand{\bcrlre}{\begin{crlre}}
\newcommand{\bxmpl}{\begin{xmpl}}
\newcommand{\brmrk}{\begin{rmrk}\rm}

\newcommand{\edfn}{\end{dfn}}
\newcommand{\ethm}{\end{thm}}
\newcommand{\elmma}{\end{lmma}}
\newcommand{\eppsn}{\end{ppsn}}
\newcommand{\ecrlre}{\end{crlre}}
\newcommand{\exmpl}{\end{xmpl}}
\newcommand{\ermrk}{\end{rmrk}}

\def \qed { \mbox{}\hfill
$\Box$\vspace{1ex}}

\title{Sections and Chapters}

\begin{document}
	
	
	\author{Keshab Chandra Bakshi and C. Silambarasan}

	\title{Structural and Dynamical Properties of Subfactor Commuting Squares}

	\maketitle
	
	
\begin{abstract}
This paper explores the algebraic invariants and ergodic behavior of the finite-index inclusions forming commuting squares of $\text{II}_1$ factors. We begin by establishing how regularity, intermediate-subfactor lattices, and Weyl groups transfer across the commuting square under relative commutant and irreducibility conditions. Furthermore, we provide a complete characterization, via a novel factorization theorem, of the unitary normalizers of the upper inclusion when the lower inclusion is regular. Applying these findings to crossed-product inclusions induced by discrete groups, we introduce a relative eigenbasis property that preserves regular inclusions. 
\end{abstract}

	\bigskip
	
	{\bf AMS Subject Classification No.:} {\large 46}L{\large 37}\,, {\large 46}L{\large 10}\,, {\large 46}L{\large 55}\,.
	
	\bigskip
	\hypersetup{linkcolor=blue}

\section{Introduction}

In the study of finite-index subfactor inclusions $N \subseteq M$, invariants such as the lattice of intermediate subfactors $\mathcal{L}(N \subseteq M)$ and the group-theoretic data provided by unitary normalizers (specifically the Weyl group $\mathcal{N}_M(N)/\mathcal{U}(N)$) encode profound structural information. Recently, the study of intermediate operator algebra inclusions has seen a flurry of activity, leading to a renewed appreciation for how these lattices capture fundamental subfactor properties. Normalizers and the Weyl group naturally give rise to Galois-type correspondences for intermediate subfactors. The structure of intermediate subfactor lattices and their interaction with commuting and co-commuting square conditions have been investigated by Watatani \cite{Wat}. These lattice-theoretic perspectives provide useful motivation for studying how intermediate subfactors behave under the commuting-square framework considered here.
In this context, crossed-product inclusions provide a natural and particularly illustrative class of examples.
For instance, given a group $K$ and a subgroup $H$, the intermediate subfactor lattice $\mathcal{L}(N \rtimes H \subseteq N \rtimes K)$ precisely mirrors the subgroup lattice $\mathcal{L}(H \subseteq K)$. This allows one to view the intermediate von Neumann subalgebra lattice as a natural generalization of the subgroup lattice. When $H$ is trivial, the resulting inclusion $N \subseteq N \rtimes K$ is \textit{regular} (meaning its normalizer $\mathcal{N}_M(N)$ generates the overalgebra). Conversely, any irreducible regular inclusion $N \subseteq M$ takes this crossed-product form, where $K$ is the Weyl group of the inclusion.

If we consider an inclusion $N \subseteq M$ and a discrete group $G$ acting on $M$ while leaving $N$ invariant, we obtain a new crossed-product subfactor $N \rtimes G \subseteq M \rtimes G$. In general, this new inclusion need not be regular even if $N\subseteq M$ is regular. More broadly, calculating the standard invariant of $N \rtimes G \subseteq M \rtimes G$ from that of $N \subseteq M$ remains a problem of significant interest in the literature (see \cite{kawa,P,CK}, for instance). These nested inclusions naturally form a \textit{non-degenerate commuting square} with respect to their canonical conditional expectations:
$$
\begin{matrix}
N\rtimes G &\subseteq & M\rtimes G\cr
\rotatebox{90}{$\subseteq$} &\ &\rotatebox{90}{$\subseteq$}\cr
N &\subseteq & M
\end{matrix}.
$$
Commuting squares of finite-index $\text{II}_1$ factors, extensively developed by Popa \cite{Popa}, provide a powerful framework for analyzing how structural properties transfer between related subfactors. Beyond their static algebraic properties, commuting squares offer a natural setting for exploring ergodic theory within operator algebras. As Kawahigashi demonstrated with paragroup actions \cite{kawa}, the structural machinery of commuting squares naturally generalizes to encapsulate crossed products by discrete groups. This establishes a vital bridge between purely algebraic invariants and ergodic theory, highlighting how the dynamical properties of group actions govern the underlying von Neumann algebraic structure.

Furthermore, commuting squares are intimately connected to the broader planar algebraic framework. The first author and Kodiyalam \cite{BV} constructed planar algebras associated with commuting squares, yielding a  diagrammatic calculus for analyzing these lattices. A critical structural concept in this context is Popa's notion of a \textit{smooth} commuting square:
$$
\begin{matrix}
Q &\subseteq & M\cr
\rotatebox{90}{$\subseteq$} &\ &\rotatebox{90}{$\subseteq$}\cr
N &\subseteq & P
\end{matrix}.
$$
Defined by the nested relative commutant condition $N^\prime \cap P_n \subseteq Q^\prime \cap M_n$ at each level $n$ of the basic construction tower, Popa introduced smoothness \cite{Popa} to ensure well-behaved asymptotic properties of the higher relative commutants. Under this condition, it has been shown in \cite{BV} that the planar algebra associated with the lower inclusion $N \subseteq P$ naturally realizes itself as a planar subalgebra of the upper inclusion $Q \subseteq M$. Conversely, applying the Guionnet--Jones--Shlyakhtenko (GJS) construction to an inclusion of planar algebras yields a commuting square of subfactors that precisely realizes this relationship.

The present article takes a complementary approach to \cite{BV} by exploring the algebraic invariants and ergodic behavior of the finite-index inclusions that constitute a commuting square. We investigate how regularity, lattice structures and Weyl groups transfer between the lower and upper inclusions. Our analysis is conducted in two distinct settings: first, under the smoothness condition $N^\prime \cap P_1 \subseteq Q^\prime \cap M_1$, and second, under a strict irreducibility assumption, namely $N^\prime \cap M = \mathbb{C}$. 

Motivated by the foundational works of Packer \cite{Pac1,Pac2} and Bannon et al.\ \cite{B}, a central structural result of this paper is a novel \textbf{normalizer factorization theorem} for irreducible commuting squares. Specifically, we prove the following:

\begin{thm}[See Theorem \ref{structure of unitaries}]
Consider a non-degenerate commuting square of $II_1$-factors
$$
\begin{matrix}
Q &\subseteq & M\cr
\rotatebox{90}{$\subseteq$} &\ &\rotatebox{90}{$\subseteq$}\cr
N &\subseteq & P
\end{matrix},
$$
where $N\subseteq P$ is a finite-index regular inclusion and $N'\cap Q=\mathbb{C}$. Then, for any $w \in \mathcal{N}_{M}(Q),$
there exist $u \in \mathcal{N}_P(N)$, $c\in \mathcal U(N'\cap M)$ and $v \in \mathcal{U}(Q)$ such that $w = ucv$.
\end{thm}

A direct consequence of this factorization (see Corollary \ref{weylcrlre}) is that, if in addition $N'\cap M=\mathbb{C}$, then the Weyl group of the upper inclusion, $W(Q \subseteq M)$, embeds as a subgroup of the lower inclusion's Weyl group, $W(N \subseteq P)$.

We apply these structural findings to the dynamical setting of crossed-product inclusions. For a regular, finite-index subfactor $N \subseteq M$ equipped with a compatible action of a discrete group $G$, we investigate the conditions under which regularity transfers to the induced inclusion $N \rtimes G \subseteq M \rtimes G$. To this end, we introduce a \textbf{relative eigenbasis property} (Definition \ref{eigenbasis}) motivated by the notion of relative elementary spectrum in ergodic theory (see for instance \cite{Pac2}) and demonstrate that it provides a necessary and sufficient mechanism for this transfer:

\begin{thm}[See Theorem \ref{regularity of crossed product}]

Let $(N\subseteq M,G,\sigma,\tau)$ be a $W^*$-dynamical extension such that $N'\cap(M\rtimes G)=\mathbb{C}$ and $N\subseteq M$ is a finite-index regular inclusion of $II_1$-factors. Then, the following conditions are equivalent : \begin{enumerate}
    \item The inclusion $N\rtimes G \subseteq M\rtimes G$ is regular,
    \item The  $W^*$-dynamical extension system  $(N\subseteq  M,G,\sigma,\tau)$ has relative eigenbasis property,
    \item[]If $G$ is finite, then the above conditions are equivalent to the following:
    \item The map $\Phi:\frac{\mathcal N_{M}(N)}{\mathcal U(N)}\times G
\longrightarrow
\frac{\mathcal N_{M\rtimes G}(N)}{\mathcal U(N)},  \quad   \Phi([v],g)=[\sigma_g(v)u_g]$
is a group isomorphism.
\end{enumerate}
\end{thm}


Finally, we investigate the lattice of intermediate subfactors of the crossed-product inclusion. Given a discrete group $G$, Popa \cite{P} (and independently Choda and Kosaki \cite{CK}) introduced the notion of a strongly outer action on a subfactor to classify actions and analyze the standard invariant of the crossed-product inclusion. Using this framework, we describe the lattice of the crossed product in terms of the original inclusion:

\begin{thm}[See Theorem \ref{strongly outer}]\label{strongly outer}
Suppose that $G$ acts strongly outer on the finite-index irreducible inclusion of $II_1$-factors $N \subseteq M$. Then, their intermediate subfactor lattices coincide
$$
\mathcal{L}(N \rtimes G \subseteq M \rtimes G) = \mathcal{L}(N \subseteq M)^G.
$$ where $\mathcal{L}(N \subseteq M)^G :=\{ P\in \mathcal{L}(N \subseteq M) : \sigma_g(P)=P~\forall g\in G\}$ denotes the sublattice of $G$-invariant intermediate subfactors of $N\subseteq M.$
\end{thm}

In a related direction, we prove that if we have a non-degenerate smooth commuting square
$$
\begin{matrix}
Q &\subseteq & M\cr
\rotatebox{90}{$\subseteq$} &\ &\rotatebox{90}{$\subseteq$}\cr
N &\subseteq & P
\end{matrix},
$$
then, the lattice of intermediate subalgebras $\mathcal{L}(N \subseteq P)$ naturally embeds as a subset of the corresponding lattice of the upper inclusion $\mathcal{L}(Q \subseteq M)$ (see Corollary \ref{smooth}).

Taken together, these results illustrate how commuting-square techniques can successfully transfer structural information between related elementary inclusions, and how Galois-type phenomena for intermediate subfactors deeply intertwine with the group-theoretic structure arising from unitary normalizers.

\section{Preliminaries}\label{preliminaries}

\begin{dfn}
    A quadruple of tracial von Neumann algebras 
    \begin{equation}\label{eq:commuting_square}
    \begin{tikzcd}
        Q \arrow[r, hook] & M \\
        N \arrow[u, hook] \arrow[r, hook] & P \arrow[u, hook]
    \end{tikzcd}
    \end{equation}
    is called a \emph{commuting square} if $E^M_PE^M_Q = E^M_QE^M_P = E^M_N$, where $E^M_A$ denotes the unique trace-preserving conditional expectation from $M$ onto a von Neumann subalgebra $A\subseteq M$. It is called \emph{non-degenerate} if, in addition, $\overline{\operatorname{sp}(PQ)} = M$, where the closure is taken with respect to the strong operator topology.
\end{dfn}
We will denote the above commuting square (see Figure \ref{eq:commuting_square}) by $(N\subseteq P,Q\subseteq M).$

The following proposition provides an equivalent characterization of a non-degenerate commuting square. For the proof and other equivalent conditions, see \cite[Proposition 1.1.5]{Popa} .
\begin{ppsn}
A commuting square $(N\subseteq P,Q\subseteq M)$ is non-degenerate if and only if any orthonormal basis of $P$ over $N$ is an orthonormal basis of $M$ over $Q$.
\end{ppsn}
Note that if $(N\subseteq P,Q\subseteq M)$ is a non-degenerate commuting square of $II_1$-factors, then $[P:N]=[M:Q].$\\

\begin{dfn}\cite{Popa} \label{definition smooth}
 A non-degenerate commuting square $(N\subseteq P, Q\subseteq M)$ of $II_1$-factors is called \textbf{smooth} if $$N'\cap P_n\subseteq Q'\cap M_n \quad \forall ~n\in \mathbb{N},$$
where $P_n$ and $M_n$ denote the $n$-th algebras in the basic construction towers for $N \subseteq P$ and $Q \subseteq M$, respectively.
\end{dfn}

\begin{dfn}
    The Weyl group of an inclusion of $II_1$- factors, $N\subseteq M$, is defined as the quotient group $$G:=\frac{\mathcal{N}_M(N)}{\mathcal{U}(N)}.$$
\end{dfn}

\begin{ppsn} \cite{Loi}
    For any two elements $v,u\in \mathcal{N}_M(N)$, $E_N(vu^*)=0=E_N(v^*u)$ if $[v]\neq[u]$ in the Weyl group $G$ of the irreducible subfactor $N\subseteq M.$
\end{ppsn}

\begin{dfn}
     A $W^*$-dynamical system is a quadruple $(M,G,\sigma,\tau)$ consisting of a tracial von Neumann algebra $(M,\tau)$, together with a strongly continuous action $\sigma$ of a countable discrete group $G$ on $M$ by $\tau$- preserving automorphisms.
\end{dfn}

\textbf{Crossed-Product von Neumann Algebra:} Given the $W^*-$ dynamical system $(M,G,\sigma,\tau)$ the crossed-product von Neumann algebra $M \rtimes G$ is realized on the tensor product Hilbert space $L^2(M) \otimes \ell^2(G)$, where $\{\delta_h\}_{h \in G}$ forms the canonical orthonormal basis of $\ell^2(G)$.
On this space, a faithful $*$-representation $\pi: M \to B(L^2(M) \otimes \ell^2(G))$ is given by
$$\pi(x)(\xi \otimes \delta_h) = \sigma_h^{-1}(x)\xi \otimes \delta_h, \quad x \in M, \, \xi \in L^2(M), \, h \in G,$$
and a unitary representation $u: G \to U(L^2(M) \otimes \ell^2(G))$ is defined by
$$u_g(\xi \otimes \delta_h) = \xi \otimes \delta_{gh}, \quad g, h \in G, \, \xi \in L^2(M).$$
Together, these operators obey the covariance relation
$$u_g \pi(x) u_g^* = \pi(\sigma_g(x)) \quad \text{for all } x \in M \text{ and } g \in G.$$

The crossed product $M \rtimes G$ is defined as the von Neumann subalgebra of $B(L^2(M) \otimes \ell^2(G))$ generated by $\pi(M)$ together with $\{u_g : g \in G\}$. Equivalently, it is the $w^*$-closure of the $*$-algebra consisting of finite linear combinations $\sum_{g \in F} \pi(x_g)u_g$ for finite subsets $F \subset G$ and elements $x_g \in M$, where operators of the form $\pi(x)u_g$ are typically written simply as $x u_g$.

\begin{dfn}
   A $W^*$-dynamical system in which the action $\sigma$ of a countable discrete group $G$ on a von Neumann algebra $M$ leaves a von Neumann subalgebra $N\subseteq M$ invariant, will be called a \textbf{$W^*$-dynamical extension system} and denoted by a quadruple $(N\subseteq M,G,\sigma,\tau)$. 
\end{dfn}

Let $N\subseteq M$ be an inclusion of $II_1$-factors, and we set $$Aut(M,N)=\{\theta\in Aut(M); \theta(N)=N\}.$$

 We recall the notion of a strongly outer action introduced by Popa (independently by Choda and Kosaki).

 \begin{dfn}\cite{CK, P, Hi}\label{strongly outer}
    An automorphism  $\theta\in \text{Aut}(M,N)$ is said to be strongly outer if we have no non-zero $a\in \cup_kM_k$ satisfying $ax=\theta(x)a~~\forall x\in N$ where $M_k$ is the $k^{th}$ algebra in the basic construction tower of $N\subseteq  M.$ The action  $\sigma$ of a group $G$ on $N\subseteq M$ is said to be strongly outer if $\sigma_g$ is strongly outer for all $g\in G$ except for the identity $\mathbf{1}_G$.
\end{dfn}

\newsection{Regularity and Intermediate Subalgebras in Commuting\\ Squares }

In this section, we study the  non-degenerate commuting squares $(N\subseteq P,Q\subseteq M)$ satisfying the condition $N'\cap P_1\subseteq Q'\cap M_1$. We investigate how this condition relates the structural properties of the horizontal inclusions $N\subseteq P$ and $Q\subseteq M$. In particular, we show that, under suitable irreducibility assumptions, regularity of $N\subseteq P$ implies the regularity of $Q\subseteq M$. We also establish a relationship between the lattices of intermediate subalgebras of $N\subseteq P$ and $Q\subseteq M$. \\

\noindent\textbf{Notation:} Let $\mathcal{N}\subseteq \mathcal{M}$ be a finite index inclusion of $II_1$-factors. We denote by $\overline{\mathcal{P}}(\mathcal{N}'\cap \mathcal{M}_1)$, the set of projections $p\in \mathcal{N}'\cap \mathcal{M}_1$ such that $E_{\mathcal{M}}(p)=[\mathcal{M}:\mathcal{N}]^{-1}\cdot1$, where $\mathcal{M}_1$ denotes the basic construction of $\mathcal{N}\subseteq \mathcal{M}$

\begin{lmma}\label{projection-inclusion}
Let $(N \subseteq P, Q \subseteq M)$ be a non-degenerate commuting square of $\mathrm{II}_1$ factors satisfying $N' \cap P_1 \subseteq Q' \cap M_1$, where $P_1 = \langle P, e_N^P \rangle$ and $M_1 = \langle M, e_Q^M \rangle$ denote the basic constructions of $N \subseteq P$ and $Q \subseteq M$, respectively. Then,
\[
\overline{\mathcal{P}}(N' \cap P_1) \subseteq \overline{\mathcal{P}}(Q' \cap M_1).
\]
\end{lmma}

\begin{proof}
Since $(N \subseteq P, Q \subseteq M)$ is a non-degenerate commuting square, the basic construction square $(P \subseteq P_1, M \subseteq M_1)$ is also a commuting square \cite[Proposition 1.1.6]{Popa}. 

Let $p \in \overline{\mathcal{P}}(N' \cap P_1)$. By the hypothesis $N' \cap P_1 \subseteq Q' \cap M_1$, we immediately have $p \in Q' \cap M_1$. 

Furthermore, since $p \in \overline{\mathcal{P}}(N' \cap P_1)$, its trace-preserving conditional expectation onto $P$ satisfies
\[
E_P(p) = [P:N]^{-1} \cdot 1.
\]
Using the commuting square property of $(P \subseteq P_1, M \subseteq M_1)$, we can compute the expectation onto $M$:
\[
E_M(p) = E_M(E_{P_1}(p)) = E_P(p) = [P:N]^{-1} \cdot 1.
\]
By the index equality for finite-index non-degenerate commuting squares, $[P:N] = [M:Q]$, which yields
\[
E_M(p) = [M:Q]^{-1} \cdot 1.
\]
Hence, $p \in \overline{\mathcal{P}}(Q' \cap M_1)$, which completes the proof.
\end{proof}

\begin{thm}\label{regularity}
 Consider a non-degenerate commuting square $(N\subseteq P,Q\subseteq M)$ of $II_1$-factors satisfying $N'\cap P_1\subseteq Q'\cap M_1$. If $N\subseteq P$ is a finite-index irreducible regular inclusion and $Q\subseteq M$ is irreducible, then $Q\subseteq M$ is regular.
\end{thm}
\begin{prf}
As $N\subseteq P$ is regular and irreducible, it follows (see for instance \cite[Theorem 3.12]{BG1}) that $[P:N]
=
\left|\frac{\mathcal{N}_P(N)}{\mathcal{U}(N)}\right|.$

Furthermore, by \cite[Proposition 1.7]{PP}  and Lemma \ref{projection-inclusion} we have, $$\left|\frac{\mathcal{N}_P(N)}{\mathcal{U}(N)}\right|
=
\operatorname{Card}\bigl(\overline{\mathcal{P}}(N'\cap P_1)\bigr)
\leq
\operatorname{Card}\bigl(\overline{\mathcal{P}}(Q'\cap M_1)\bigr)
=
\left|\frac{\mathcal{N}_M(Q)}{\mathcal{U}(Q)}\right|,$$
Therefore, $$[P:N]
=
\left|\frac{\mathcal{N}_P(N)}{\mathcal{U}(N)}\right|
\leq
\left|\frac{\mathcal{N}_M(Q)}{\mathcal{U}(Q)}\right|
\leq
[M:Q].$$

Since the commuting square is non-degenerate, we have $[P:N]=[M:Q].$
Hence,
\[
\left|\frac{\mathcal{N}_M(Q)}{\mathcal{U}(Q)}\right|
=
[M:Q].
\]

It follows from \cite[Theorem 3.4]{BG}, that the inclusion $Q\subseteq M$ is regular. \qed
\end{prf}

\begin{thm}\label{smooth lattice}
Let $(N\subseteq P,Q\subseteq M)$ be a non-degenerate commuting square of $II_1$-factors satisfying
$N^{\prime}\cap P_1\subseteq Q^{\prime}\cap M_1.$
 Then, there exists an embedding of $\mathcal{L}(N\subseteq P)$ into $ \mathcal{L}(Q\subseteq M).$
\end{thm}

\begin{proof}
Define a map
$$
\phi:\mathcal{L}(N\subseteq P)\longrightarrow \mathcal{L}(Q\subseteq M),
\qquad
\phi(K)=\{e_K^P\}'\cap M.
$$
Since $e_K^P\in N'\cap P_1\subseteq Q'\cap M_1$, we have $Q\subseteq \phi(K)\subseteq M.$ Hence $\phi$ is well defined. We now show that $\phi(K)$ is a proper intermediate von Neumann subalgebra, whenever $K$ is proper.\\
\textbf{Case I.} Suppose $\phi(K)=Q$. Then, $K=\{e_K^P\}'\cap P\subseteq \{e_K^P\}'\cap M=Q.$ Since $K\subseteq P$, it follows that $K\subseteq P\cap Q=N,$
contradicting the assumption that $K$ is a proper intermediate subalgebra of $N\subseteq P$.\\
\textbf{Case II.} Suppose $\phi(K)=M$. Then, $\{e_K^P\}'\cap M=M,$ so $e_K^P\in M'\subseteq P'$. Hence $P\subseteq \{e_K^P\}'\cap P=K,$
which implies $K=P$, again a contradiction. Therefore, $Q\subsetneq \phi(K)\subsetneq M.$\\
Next, define
$$
\psi:\mathcal{L}(Q\subseteq M)\longrightarrow \mathcal{L}(N\subseteq P),
\qquad
\psi(L)=L\cap P.
$$
Clearly, $N\subseteq \psi(L)\subseteq P,$ so $\psi$ is well defined. Moreover, $\psi(\phi(K))=(\{e_K^P\}'\cap M)\cap P=\{e_K^P\}'\cap P=K.$
Hence $\psi\circ\phi=\mathrm{id}_{\mathcal{L}(N\subseteq P)}$, and therefore $\phi$ is injective. Thus, $\mathcal{L}(N\subseteq P)$ embeds into $\mathcal{L}(Q\subseteq M).$
\end{proof}

\begin{crlre}\label{smooth}
 Let $(N\subseteq P,Q\subseteq M)$ be a non-degenerate smooth commuting  square of $II_1$-factors.
 Then, there exists an embedding of $\mathcal{L}(N\subseteq P)$ into $\mathcal{L}(Q\subseteq M)$. 
\end{crlre}

\begin{prf}
    This follows from Theorem \ref{smooth lattice} and \Cref{definition smooth}. \qed
\end{prf}

\begin{crlre}
Let $(N\subseteq P, Q\subseteq M)$ be a non-degenerate commuting square of $II_1$-factors satisfying the hypotheses of Theorem~\ref{regularity}. Then, there exists an embedding of $\mathcal{L}(W_P^N)$ into $\mathcal{L}(W_M^Q),$ where $\mathcal{L}(W_P^N)$ and $\mathcal{L}(W_M^Q)$ denote the subgroup lattices of the Weyl groups of the inclusions $N\subseteq P$ and $Q\subseteq M$, respectively.
\end{crlre}

\begin{proof}
By Theorem~\ref{regularity}, the inclusion $Q\subseteq M$ is regular. Since both $N\subseteq P$ and $Q\subseteq M$ are finite index irreducible regular inclusions of $II_1$-factors, by \cite[Theorem 3.1]{Hong}, the inclusions $N\subseteq P$ and $Q\subseteq M$ are isomorphic to $N\subseteq N\rtimes W_P^N$ and $Q\subseteq Q\rtimes W_M^Q$ respectively. Thus, we get
\[
\mathcal{L}(N\subseteq P)\cong \mathcal{L}(W_P^N)
\quad\text{and}\quad
\mathcal{L}(Q\subseteq M)\cong \mathcal{L}(W_M^Q).
\]
The conclusion now follows from Theorem~\ref{smooth lattice}, which provides an injective map $\mathcal{L}(N\subseteq P)\longrightarrow \mathcal{L}(Q\subseteq M).$
\end{proof}

\newsection{Commuting Squares and Crossed-Product Inclusions}

In this section, we study the regularity and lattice structure of non-degenerate commuting squares $(N\subseteq P,Q\subseteq M)$. We first analyze the structure of normalizing unitaries of the upper inclusion and use this to obtain relations between the corresponding Weyl groups. We then apply these results to crossed-product inclusions arising from $W^*$-dynamical extension systems. In particular, we characterize the regularity of the crossed-product inclusion in terms of the relative eigenbasis property. Finally, we discuss the behavior of intermediate subfactors under strongly outer actions.
\medskip

\begin{thm}\label{structure of unitaries}
Consider a non-degenerate commuting square $(N \subseteq P,Q\subseteq M)$ of $II_1$-factors, where $N\subseteq P$ is a finite-index regular inclusion and $N'\cap Q=\mathbb{C}$. Then, for any $w \in \mathcal{N}_{M}(Q),$
there exist $u \in \mathcal{N}_P(N)$, $c\in \mathcal U(N'\cap M)$ and $v \in \mathcal{U}(Q)$ such that $w = ucv$.
\end{thm}

\begin{prf} 
 Let $\{\zeta_i\}_{i=1}^n \subseteq \mathcal{N}_P(N)$ be a unitary orthonormal basis for the inclusion $N \subseteq P$, where $n=[P:N].$ (see for instance \cite{BG1}). Then, $\{\zeta_i\}_{i=1}^n$ also forms a unitary orthonormal basis for $Q\subseteq M$. Thus, for any  $w \in \mathcal{N}_M(Q) $, we may write $w = \sum_{i=1}^n \zeta_i w_i,$ where $ w_i := E_{Q}(\zeta_i^* w).$

Note that $\zeta_i x \zeta_i^* \in N \subseteq Q$ for $x \in N$, as  $\zeta_i \in \mathcal{N}_P(N)$. Moreover, as $w \in \mathcal{N}_{M}(Q)$, it follows that $w^* \zeta_i x \zeta_i^* w \in Q$.
Applying the conditional expectation $E_{Q }$, we obtain
\[
x E_{Q}(\zeta_i^* w) = E_{Q}(\zeta_i^* w)\, w^* \zeta_i x \zeta_i^* w.
\]
It follows that $z_i:=w_i w^* \zeta_i \in N' \cap  M.$ Since $w_i=z_i\zeta_i^*w$, we have
$$
w_iw_i^*=z_i\zeta_i^*ww^*\zeta_iz_i^*=z_iz_i^*\in N'\cap Q=\mathbb{C}.
$$

 Thus there exists $\alpha_i\geq 0$ such that $z_iz_i^*=\alpha_11.$ At least one coefficient $w_j$ is nonzero, since $w=\sum_i\zeta_iw_i$ is unitary. Choose such a $j$; then $\alpha_j>0.$ Setting  $u_j := \alpha_j^{-1/2} z_j \in \mathcal{U}(N' \cap M)$ and $v := \alpha_j^{-1/2} w_j \in \mathcal{U}(Q)$ are unitaries satisfying $v = u_j \zeta_j^* w$. Thus $w=\zeta_ju_j^*v$ where $\zeta_j\in \mathcal{N}_P(N)$, $u_j\in \mathcal{U}(N'\cap M)$ and $v\in \mathcal U(Q).$ \qed

\end{prf}

\begin{crlre}\label{weylcrlre}
Let $(N\subseteq P, Q\subseteq M)$ be a non-degenerate irreducible commuting square of $II_1$ factors such that $N\subseteq P$ is a finite index regular inclusion. Then, the Weyl group of $Q\subseteq M$ is isomorphic to a subgroup of the  Weyl group of $N\subseteq P$.
\end{crlre}

\begin{prf}
Let $W$ and $W'$ be the Weyl groups $\frac{\mathcal{N}_P(N)}{\mathcal{U}(N)}$ and $\frac{\mathcal{N}_M(Q)}{\mathcal{U}(Q)}$ respectively. By Theorem \ref{structure of unitaries}, every element $w \in \mathcal{N}_M(Q)$ can be written as $w = uv$ with $u \in \mathcal{N}_P(N)$ and $v \in \mathcal{U}(Q)$, implying that each left coset in $W'$ can be represented by an element of $\mathcal{N}_P(N)$. We define the homomorphism $\phi: W' \to W$ by$$\phi([u]_{W'}) = [u]_W, \qquad u \in \mathcal{N}_P(N).$$ 
We first show that $\phi$ is well defined. Let $[u_1]_{W'}= [u_2]_{W'}$. If $[u_1]_{W} \neq [u_2]_{W}$, then, $E_N(u_1^*u_2)=0$ implying $E_Q(u_1^*u_2)=E_QE_P(u_1^*u_2)=E_N(u_1^*u_2)=0$, which is a contradiction (see for instance \cite[Lemma 2.12.]{BG1}). Thus, the map $\phi$ is well-defined.
This map is a homomorphism as, $$\phi([u]_{W'}[v]_{W'})=\phi([uv]_{W'})=[uv]_{W}=[u]_W[v]_W=\phi([u]_{W'})\phi([v]_{W'}).$$
We claim that this map is injective. Let $\phi([u_1]_{W'})=\phi([u_2]_{W'})$. This implies $[u_1]_W=[u_2]_W$. Thus, $u_2^*u_1\in \mathcal{U}(N)$ which means $u_1=u_2v$ for some $v\in \mathcal{U}(N)$.  It follows that $[u_1]_{W'}=[u_2v]_{W'}=[u_2]_{W'}$. Hence the map is an injective homomorphism and so $W'$ is a subgroup of $W$. \qed
 
\end{prf}

\subsection{Dynamical Properties of Crossed-Product Inclusions}

Prior to investigating the interplay between group dynamics and von Neumann algebraic inclusions, we summarize several key definitions regarding $W^*$-dynamical systems for the reader's convenience. Throughout this section $G$ is assumed to be a discrete group.

\begin{dfn} \cite{Popa2}
Let $N$ be a von Neumann algebra and $\sigma:G\rightarrow Aut(N)$ an action of $G$ on $N$.
    A 1-cocycle for the action $\sigma$ is a map $\lambda:G\rightarrow \mathcal{U}(N)$ satisfying
    $$\lambda(gh)=\lambda(g)\sigma_g(\lambda(h))\quad \forall ~g,h\in G.$$
    \end{dfn}

Motivated by  the notion of relative eigenvalues and eigenfunctions for group-measure von Neumann algebras in  \cite{Pac2}, we introduce the following definition.

\begin{dfn}
    A 1-cocycle $\lambda:G\rightarrow \mathcal{U}(N)$ is called the relative eigenvalue of $(N\subseteq M,G,\sigma,\tau)$ with relative eigenvector $u\in \mathcal{U}(M)$ if 
    $$\sigma_g(u)=u\lambda(g)\quad \forall~g\in G.$$

\end{dfn}

\begin{thm} \label{Analytical Property}

Let $(N\subseteq M,G,\sigma,\tau)$ be a $W^*$-dynamical extension system. Then,  $u\in\mathcal{U}(M)\cap\mathcal{N}_{M\rtimes G}(N\rtimes G)$ if and only if $u\in\mathcal{N}_M(N)$ and $u$ is a relative eigenvector of the action.
\end{thm}

\begin{prf}Suppose that $u\in \mathcal{U}(M)\cap\mathcal{N}_{M \rtimes G}(N\rtimes G)$. For any $n\in N$ and $g\in G$, we have $u(nu_g)u^*=un\sigma_g(u^*)u_g\in N\rtimes G$. Taking $g=\mathbf{1}_G$, it follows that $unu^*\in N\quad \forall n\in N$. Hence, $u\in \mathcal{N}_M(N).$ 

Moreover, for each $g\in G$, since $u_g \in \mathcal{U}(N\rtimes G)$, we obtain $u\sigma_g(u^*)u_g\in \mathcal{U}(N\rtimes G)$. Thus, there exists $v_g \in \mathcal{U}(N)$ such that $u  \sigma_g(u^*) = v_g.$

Define, $\lambda(g) := u^* v^*_g u $. Then, $\lambda(g) \in \mathcal{U}(N)$, and a direct computation shows that $\lambda(g) = u^* \sigma_g(u).$
Furthermore, $$\lambda(gh) = u^* \sigma_{gh}(u) = u^* \sigma_g(\sigma_h(u)) 
= u^* \sigma_g(u \lambda(h)) 
= u^* \sigma_g(u) \, \sigma_g(\lambda(h)) 
= \lambda(g)\, \sigma_g(\lambda(h)).$$
Thus, $\lambda : G \to \mathcal{U}(N)$ is a $1$-cocycle, and $u$ is a relative eigenvector for the action with relative eigenvalue $\lambda$. \\
\noindent
Conversely, if $u$ is a relative eigenvector for the action, then there exists a 1-cocycle $\lambda:G\to \mathcal{U}(N)$
such that $\sigma_g(u)=u\lambda(g),\forall g\in G.$
Hence, for $x=\sum_{g\in G}n_gu_g\in N\rtimes G$, we have
\[
u^*xu
=
\sum_{g\in G} u^* n_g u\,\lambda(g)\,u_g .
\] Because $u\in \mathcal N_M(N)$, we have $u^* n_g u\,\lambda(g)\in N ~ \text{for all } g\in G.$ It follows that $u^*xu \in N\rtimes G.$ Thus, $u\in \mathcal N_{M\rtimes G}(N\rtimes G).$\qed
\end{prf}

\begin{crlre}
Let $(N\subseteq M,G,\sigma,\tau)$ be a $W^*$-dynamical extension system such that \\$N'\cap(M\rtimes G)=\mathbb{C},$ and suppose that $N\subseteq M$ is a finite-index regular
inclusion of $II_1$-factors. The action $\sigma$ induces an
action $\alpha:G\to\operatorname{Aut}(W)$ on the Weyl group $W:=\mathcal{N}_M(N)/\mathcal{U}(N),$ 
$\alpha_g([u]):=[\sigma_g(u)].$ Then $W':=\mathcal{N}_{M\rtimes G}(N\rtimes G)/\mathcal{U}(N\rtimes G)$ is isomorphic to $W^G,$ where $W^G:=\{[u]\in W:\alpha_g([u])=[u]
       \text{ for every }g\in G\}.$
\end{crlre}

\begin{prf}
By Corollary~\ref{weylcrlre}, every element of $W'$ admits
a representative $$u\in\mathcal{N}_M(N)\cap
\mathcal{N}_{M\rtimes G}(N\rtimes G),$$ and the assignment $$\phi:W'\longrightarrow W,
\qquad
\phi([u]_{W'})=[u]_W,$$ is a well-defined injective homomorphism.
We show that $\operatorname{Im}\phi=W^G$.

Let $[u]_W\in\operatorname{Im}\phi$, with $u$ chosen as above.
By Theorem~\ref{Analytical Property}, there exists a
$1$-cocycle $\lambda:G\to\mathcal{U}(N)$ such that
$\sigma_g(u)=u\lambda(g)$ for every $g\in G$. Hence
\[
\alpha_g([u]_W)
=[\sigma_g(u)]_W
=[u\lambda(g)]_W
=[u]_W,
\]
so $\operatorname{Im}\phi\subseteq W^G$. Conversely, let $[v]_W\in W^G$. Then $\lambda(g):=v^*\sigma_g(v)\in\mathcal{U}(N).$ Moreover, $\lambda(gh)=\lambda(g)\sigma_g(\lambda(h)),$
so $\lambda$ is a $1$-cocycle and
$\sigma_g(v)=v\lambda(g)$. Since
$v\in\mathcal{N}_M(N)$, Theorem~\ref{Analytical Property}
gives $v\in\mathcal{N}_{M\rtimes G}(N\rtimes G).$
Thus $[v]_W=\phi([v]_{W'})$, proving
$W^G\subseteq\operatorname{Im}\phi$.
Therefore $\phi$ induces the desired isomorphism.
\qed
\end{prf}

\begin{dfn}\label{eigenbasis}
The $W^*$-dynamical extension system $(N\subseteq M,G,\sigma,\tau)$ is said to have the \textit{relative eigenbasis property} if there exists a unitary Pimsner--Popa basis for $M$ over $N$ consisting of relative eigenvectors of the action that belong to $\mathcal{N}_M(N)$.
\end{dfn}

\noindent
The following theorem characterizes the regularity of crossed product inclusions by establishing several equivalent conditions.

\begin{thm}\label{regularity of crossed product}
Let $(N\subseteq M,G,\sigma,\tau)$ be a $W^*$-dynamical extension such that $N'\cap(M\rtimes G)=\mathbb{C}$ and $N\subseteq M$ is a finite-index regular inclusion of $II_1$-factors. Then, the following conditions are equivalent : \begin{enumerate}
    \item The inclusion $N\rtimes G \subseteq M\rtimes G$ is regular,
    \item The  $W^*$-dynamical extension system  $(N\subseteq  M,G,\sigma,\tau)$ has relative eigenbasis property,
    \item[]If $G$ is finite, the above conditions are also equivalent to the following:
    \item The map $$\Phi:\frac{\mathcal N_{M}(N)}{\mathcal U(N)}\times G
\longrightarrow
\frac{\mathcal N_{M\rtimes G}(N)}{\mathcal U(N)},  \quad   \Phi([v],g)=[\sigma_g(v)u_g]$$
is a group isomorphism.
\end{enumerate}
\end{thm}
\begin{proof} $(1)\implies (2): $ 

Assume that the inclusion $N \rtimes G \subseteq M \rtimes G$ is regular, and let $\{w_i\}_{i=1}^n \subseteq \mathcal{N}_{M \rtimes G}(N \rtimes G)$ represent the left cosets in the Weyl group of  $N\rtimes G\subseteq M\rtimes G$ where $n=[M:N]$. 
Since $(N\subseteq M, N\rtimes G\subseteq M\rtimes G)$ is a non-degenerate commuting square,  Theorem \ref{structure of unitaries} ensures that  each representative admits a factorization 
$$w_i = u_i v_i \quad \text{for some } u_i \in \mathcal{U}(M) \text{ and } v_i \in \mathcal{U}(N \rtimes G).$$ 
Because $w_i$ and $u_i$ define the same equivalence class $[w_i] = [u_i]$ modulo $\mathcal{U}(N \rtimes G)$, we may replace each $w_i$ with $u_i$. Consequently, we can select a system of coset representatives $\{u_i\}_{i=1}^n$ contained entirely within $\mathcal{U}(M)$ which, by  Theorem \ref{Analytical Property}, are relative eigenvectors of the action and belong to $\mathcal N_M(N)$.\\
  Furthermore, $\{u_i\}_{i=1}^n$ forms a Pimsner--Popa basis of $M\rtimes G$ over $N\rtimes G$. Hence, for every $x\in M\subseteq M\rtimes G$, we have $$x=\sum_{i=1}^n E_{N\rtimes G}(xu_i^*)\,u_i .$$ As $E_{N\rtimes G}E_M=E_N$, $$x=\sum_{i=1}^n E_N(xu_i^*)\,u_i.$$

  This demonstrates that $\{u_i\}_{i=1}^n$ is a unitary orthonormal basis of $M$ over $N$ consisting of eigenvectors lying in $\mathcal N_M(N)$. Therefore, $(N\subseteq M,G,\sigma,\tau)$ has relative eigenbasis property.\\
$(2)\implies (1):$

If $(N\subseteq M,G,\sigma,\tau)$ has relative eigenbasis property, then, they also form the basis of the inclusion $N\rtimes G\subseteq M\rtimes G$. By Theorem \ref{Analytical Property}, these elements yield normalizing basis of the inclusion $N\rtimes G \subseteq M\rtimes G$.
Hence, the inclusion is regular.\\ 
$(2)\implies (3):$

Since the inclusion $N\subseteq M$ has relative eigenbasis $\{u_i\}_i$, choose $\{u_i\}_i$ as left coset representative of the Weyl group of $N\subseteq M.$ 
 Then the map $\Phi$ can be equivalently defined as $\Phi([u_i],g)=[\sigma_g(u_i)u_g].$
We claim that $\Phi$ is an injective group homomorphism. Let $u_1,u_2\in\mathcal N_M(N)$ and $g_1,g_2\in G$. Then,  \begin{align*}
\Phi(([u_k],g_k)([u_l],g_l))&=\Phi(([u_ku_l],g_kg_l))\\
                            &=[\sigma_{g_kg_l}(u_ku_l)u_{g_kg_l}]\\
                            &=[u_{g_k}u_{g_l}u_ku_l]\\
                            &=[u_{g_k}\sigma_{g_l}(u_k)u_{g_l}u_l]\\
                            &=[u_{g_k}u_k\lambda(g_l)][u_{g_l}u_l]\\
                            &=[u_{g_k}u_k][u_{g_l}u_l]=[\sigma_{g_k}(u_k)u_{g_k}][\sigma_{g_l}(u_l)u_{g_l}].
\end{align*}
If $\Phi([u],g)=[1]$, then, $\sigma_g(u)u_g\in\mathcal{U}(N).$ Since $\sigma_g(u)\in\mathcal{U}(M)$, it follows that \\
$u_g=\sigma_g(u)^*(\sigma_g(u)u_g)\in M.$
Hence $u_g\in M$, which is possible only if $g=\mathbf{1}_G$. Consequently, $\Phi([u],\mathbf{1}_G)=[u]=[1],$ and therefore $u\in\mathcal{U}(N)$. Thus $[u]=[1]$, proving that $\Phi$ is injective.
As $|\frac{\mathcal N_{M}(N)}{\mathcal U(N)}\times G|
=
|\frac{\mathcal N_{M\rtimes G}(N)}{\mathcal U(N)}|,$
$\Phi$ is an isomorphism.\\
$(3)\implies (1):$

$N\subseteq M$ and $N\subseteq N\rtimes G$ regular implies $N\subseteq M\rtimes G$ regular \cite[Theorem 3.5]{BG}, so the commuting square $(N\subseteq M,N\rtimes G\subseteq M\rtimes G)$ is isomorphic to 
$$
\begin{matrix}
N\rtimes G &\subseteq & N\rtimes\frac{\mathcal N_{M\rtimes G}(N)}{\mathcal U(N)}\cr
 \rotatebox{90}{$\subseteq$} &\ &\rotatebox{90}{$\subseteq$}\cr
N &\subseteq & M
\end{matrix}.$$ 
As $\frac{\mathcal N_{M}(N)}{\mathcal U(N)}\times G~ \cong~ \frac{\mathcal N_{M\rtimes G}(N)}{\mathcal U(N)}$ and  $G$ is a normal subgroup of $\frac{\mathcal N_{M}(N)}{\mathcal U(N)}\times G$, it follows that the inclusion $N\rtimes G\subseteq M\rtimes G$ is regular.
\end{proof}

\begin{rmrk}
The triviality of the relative commutant, $N'\cap (M\rtimes G)=\mathbb{C}$, is ensured, in particular, if the action of $G$ is strongly outer on irreducible inclusion $N\subseteq M$.
\end{rmrk}

As in the previous section, we now investigate the relationship between the lattice structures of $\mathcal{L}(N\subseteq M)$ and its crossed-product counterpart $\mathcal{L}(N\rtimes G\subseteq M\rtimes G)$. Let $G$ be a finite group equipped with an outer action on $M$ that restricts to an outer action on $N$. Although the passage to crossed products does not generally  preserve the intermediate subfactor lattice, we show that a strongly outer action  rigidifies the inclusion, yielding a natural correspondence between the two lattices.

\begin{thm}\label{strongly outer}
Suppose that $G$ acts strongly outer on the finite-index irreducible inclusion
$N\subseteq M$. Then, their intermediate subfactor lattices coincide
$$
\mathcal{L}(N \rtimes G \subseteq M \rtimes G) = \mathcal{L}(N \subseteq M)^G,
$$ where $\mathcal{L}(N \subseteq M)^G :=\{ P\in \mathcal{L}(N \subseteq M) : \sigma_g(P)=P~\forall g\in G\}$ denotes the sublattice of $G$-invariant intermediate subfactors of $N\subseteq M.$
\end{thm}

\begin{prf}
Put $\widetilde{N}=N\rtimes G$ and $\widetilde{M}=M\rtimes G$, and let
$\widetilde{M}_k=M_k\rtimes G$, where $M_k$ is the $k$-th algebra in the
basic construction tower for $N\subseteq M$. Then, by \cite{CK}, the basic
construction tower for $\widetilde{N}\subseteq\widetilde{M}$ is
\[
\widetilde{N}\subseteq\widetilde{M}\subseteq\widetilde{M}_1
\subseteq\cdots\subseteq\widetilde{M}_k\subseteq\cdots.
\]

Since $N\subseteq M$ is irreducible and of finite index,
$\mathcal{L}(N\subseteq M)$ is finite \cite{Wat}. Moreover, since $\sigma$ is a
strongly outer action of $G$ on $N\subseteq M$, it follows that $N\subseteq\widetilde{M}$ is irreducible. In particular,
$\widetilde{N}\subseteq\widetilde{M}$ is irreducible, and hence
$\mathcal{L}(\widetilde{N}\subseteq\widetilde{M})$ is also finite.

Let $\widetilde{P}$ be an intermediate subfactor of
$\widetilde{N}\subseteq\widetilde{M}$, and put $P=\widetilde{P}\cap M.$
Then, $P$ is an intermediate subfactor of $N\subseteq M$. Since
$\widetilde{N}=N\rtimes G\subseteq\widetilde{P}$, the canonical unitaries
$u_g$, $g\in G$, belong to $\widetilde{P}$. For $x\in P$ and $g\in G$, we have $\sigma_g(x)=u_gxu_g^*\in\widetilde{P}.$

Since $\sigma_g$ preserves $M$, we obtain $\sigma_g(x)\in\widetilde{P}\cap M=P.$
Thus, $\sigma$ restricts to an action of $G$ on $P$. By \cite[Proposition 7]{CK} (see also \cite{P}), we have  $(\widetilde{N})'\cap\widetilde{M}_k
=
(N'\cap M_k)^\sigma.$ In particular,
\[
e_{\widetilde{P}}^{\widetilde{M}}
\in
(\widetilde{N})'\cap\widetilde{M}_1
\subseteq
N'\cap M_1.
\]
By \cite[Corollary 3.3]{BL},
\[
N'\cap M_1
=
e_M^{\widetilde{M}}
\left(
N'\cap
\left\langle
\widetilde{M},e_M^{\widetilde{M}}
\right\rangle
\right)
e_M^{\widetilde{M}}.
\]
Therefore, $e_{\widetilde{P}}^{\widetilde{M}}e_M^{\widetilde{M}}
=
e_M^{\widetilde{M}}e_{\widetilde{P}}^{\widetilde{M}}.$ Consequently, $E_{\widetilde{P}}^{\widetilde{M}}E_M^{\widetilde{M}}
=
E_M^{\widetilde{M}}E_{\widetilde{P}}^{\widetilde{M}}.$
By \cite[Theorem 3.2]{CD}, it follows that $\widetilde{P}=P\rtimes G.$
Hence every intermediate subfactor of
$\widetilde{N}\subseteq\widetilde{M}$ is of the form
$P\rtimes G$ for some $P\in\mathcal{L}(N\subseteq M)$. Conversely, for any $P\in \mathcal{L}(N \subseteq M)^G$, the action restricts to $P$, so $P\rtimes G\in \mathcal{L}(\tilde{N} \subseteq \tilde{M})$. Therefore, $\mathcal{L}(N\rtimes G\subseteq M\rtimes G)
=
\mathcal{L}(N\subseteq M)^G .$ \qed
\end{prf}

\begin{ppsn}
Suppose that $K$ is a finite group acting strongly outer on an
irreducible regular finite index inclusion of $II_1$-factors $N\subseteq M$. Then,
\[
\mathcal{L}(N\subseteq M\rtimes K)
\cong
\mathcal{L}(H\rtimes K),
\]
where $H$ denotes the Weyl group of the inclusion $N\subseteq M$.
\end{ppsn}

\begin{prf}
Since $K$ acts strongly outer, $(N\subseteq M,N\rtimes K\subseteq M\rtimes K)$ is an irreducible commuting square. As $N\subseteq M$ and
$N\subseteq N\rtimes K$ are regular inclusions, \cite[Theorem 3.5]{BG} implies that $\left(
N\subseteq M,\,
N\rtimes K\subseteq M\rtimes K
\right)
\cong
\left(
N\subseteq N\rtimes H,\,
N\rtimes K\subseteq N\rtimes G
\right),$ where $G$ is the Weyl group of the inclusion
$N\subseteq M\rtimes K$. Since $\left(
N\subseteq N\rtimes H,\,
N\rtimes K\subseteq N\rtimes G
\right)$ is a non-degenerate commuting square, we have $G=HK$ and $H\cap K=\{e\}.$ 

Moreover, $M\subseteq M\rtimes K$ is regular, and $N\rtimes H\subseteq N\rtimes G
\cong
M\subseteq M\rtimes K.$
Hence $N\rtimes H\subseteq N\rtimes G$ is regular, which implies that
$H$ is a normal subgroup of $G$   (see \cite[Corollary 3.5]{BCP}).

Therefore, as $H\trianglelefteq G$, $G=HK$, and
$H\cap K=\{e\}$, we obtain $G=H\rtimes K.$ Consequently, $\mathcal{L}(N\subseteq M\rtimes K)
\cong
\mathcal{L}(G)
\cong
\mathcal{L}(H\rtimes K),$ as required. \qed
\end{prf}
\medskip

\vspace{0.5cm}

\noindent \textbf{Concluding Remarks.}
\begin{itemize}
    \item \textit{Non-irreducible Case}: A fundamental assumption underpinning the main results of this paper, particularly the normalizer factorization and the Weyl group embedding is the irreducibility of the inclusions. The behavior of these invariants in commuting squares lacking this property is currently unknown; exploring this broader, reducible regime will be the subject of future investigations.
    \item \textit{Infinite Index Case}: We also intend to investigate whether analogous results can be established for simple $C^*$-algebra inclusions beyond the finite-index setting, where the finite-index techniques used in the present work are no longer directly applicable.
\end{itemize}

\subsection*{Acknowledgement} 
The first author gratefully acknowledges financial support from the Anusandhan National Research Foundation (ANRF) under grant ANRF/ECRG/2024/002328/PMS. All mathematical results and proofs presented in this paper were originally conceived and derived by the authors independently, without the use of AI tools. During the preparation and revision of the manuscript, ChatGPT was used as an auxiliary tool for language editing, improving mathematical exposition, and reviewing the presentation of selected proofs. The tool also helped identify an error in an earlier version of the manuscript. All mathematical arguments were thoroughly verified by the authors, who take full responsibility for the accuracy of the work.

\bigskip

\noindent {\em Department of Mathematics}\\
{\em Indian Institute of Technology Kanpur}\\
{\em Uttar Pradesh $208016$, India}
\medskip

\noindent {Keshab Chandra Bakshi:} bakshi209@gmail.com, keshab@iitk.ac.in\\
{C. Silambarasan:} silamcmath23@gmail.com, silamc23@iitk.ac.in

\end{document}